\documentclass[12pt]{amsart}
\usepackage{amsmath}	% required for `\multline' (yatex added)
\usepackage[dvips]{graphicx}
\input{amssymb.sty}
\input xy
\xyoption{all}
\theoremstyle{plain}
  \newtheorem{thm}{Theorem}[section]
  \newtheorem{lem}[thm]{Lemma}
  \newtheorem{cor}[thm]{Corollary}
  \newtheorem{prop}[thm]{Proposition}
  \newtheorem{fact}[thm]{Fact}
\theoremstyle{definition}
  \newtheorem{defn}[thm]{Definition}

  \newtheorem{notation}{Notation\!\!}

\theoremstyle{remark}
  
\newtheorem{acknowledgment}{Acknowledgment}

\newcommand{\Ktx}{K_{\tilde{X}}}
\newcommand{\Otx}{{\mathcal O}_{\tilde{X}}}
\newcommand{\tr}{\operatorname{tr}}

\newcommand{\bR}{{\bold R}}
\newcommand{\CC}{{\mathbb{C}}}
\newcommand{\ch}{\operatorname{ch}}

\newcommand{\ext}{\operatorname{ext}}
\newcommand{\Ext}{\operatorname{Ext}}

\newcommand{\Hom}{\operatorname{Hom}}

\newcommand{\IIm}{\operatorname{Im}}
\newcommand{\Ker}{\operatorname{Ker}}

\newcommand{\NN}{{\mathbb{N}}}
\newcommand{\NS}{\operatorname{NS}}

\newcommand{\Ox}{{\mathcal O}_X}

\newcommand{\QQ}{{\mathbb{Q}}}

\newcommand{\Spec}{\operatorname{Spec}}

\newcommand{\ZZ}{{\mathbb{Z}}}

\begin{document}
\bibliographystyle{amsplain}
\title[moduli of sheaves on Enriques surfaces]
{Singularities and Kodaira dimension of moduli scheme of stable sheaves
on Enriques surfaces}
\author{Kimiko Yamada}
\email{yamada@xmath.ous.ac.jp}
\address{Department of applied mathematics, Faculty of Science, 
Okayama University of Science, Japan}
\subjclass{Primary 14J60; Secondary 14B05, 14D20, 14Exx}
\thanks{This work was supported by the Grants-in-Aid for 
Young Scientists (B), JSPS, No. 23740037. }
\begin{abstract}
Let $M$ be a moduli scheme of stable sheaves with fixed Chern classes on
an Enriques surface or a hyper-elliptic surface. 
If its expected dimension is 7 or more, then $M$ admits only 
canonical singularities.
Moreover, if $M$ is compact, then its Kodaira dimension is zero.
%
%We shall show that the Kodaira dimension of the moduli scheme of stable sheaves 
%on any Enriques surface is zero when its expected dimension is 7 or more.
%A similar result holds also for bi-elliptic surfaces.
%We prove this by estimating the cup-product structre of $\Ext^1$.
\end{abstract}

\maketitle
\section{Introduction}
Let $X$ be a non-singular projective minimal surface over $\CC$, and
$H$ an ample line bundle on $X$.
There is a coarse moduli scheme $M(H)$ (resp. $\bar{M}(H)$) 
of $H$-stable (resp. $H$-semistable) sheaves with fixed Chern classes 
$(r,c_1,c_2)\in \ZZ_{>0}\times \NS(X)\times \ZZ$, where we assume $r>1$.
Gieseker and Maruyama (\cite{Gi:moduli}, \cite{Ma:moduli2}) proved that
$\bar{M}(H)$ is projective over $\CC$, and $M(H)$ is its open subscheme.
Suppose that $X$ is a minimal surface with $\kappa(X)=0$.
Then $K_X$ is numerically equivalent to $0$, and
there is a covering map $\pi: \tilde{X}\rightarrow X$,
where $\tilde{X}$ is a non-singular projective surface with 
$K_{\tilde{X}}={\mathcal O}_{\tilde{X}}$.
Let $d$ be the
degree of $\pi$. The main result of this article is as follows.
\begin{thm}\label{thm:main}
Let $X$ be a minimal surface with $\kappa(X)=0$.
%$H$ is a $(r,c_1,c_2)$-generic...
Suppose that 
\begin{align*}
2rc_2-(r-1)c_1^2-r^2\chi(\Ox) & \geq 3d, \quad \text{that is,} \\
\operatorname{exp dim} M(H) = \ext^1(E,E) -\ext^2(E,E)^{\circ}  & \geq 3d+1+p_g(X).
\end{align*}
%
%-- ver.3
%\[\operatorname{exp dim} M(H)=\ext^1(E,E)^{\circ}-\ext^2(E,E)^{\circ}
%  \geq 3d+\chi(\Ox).\]
%-- ver. 2
%the expected dimension of $M(H)$, that is,
%$\ext^1(E,E)-\ext^2(E,E)$ is greater than $d(d^2+2)$.
%$2rc_2-(r-1)c_1^2\geq r^2\chi(\Otx)+ d(d^2+2)$.
%
%-- ver.1
% 2rc_2-(r-1)c_1^2\geq r^2\chi({\mathcal O}_{\tilde{X}}) + d(r^2+3)
%
Then the moduli scheme $M(H)$ of $H$-stable sheaves with Chern classes
$(r,c_1,c_2)$ is  l.c.i., normal, 1-Gorenstein, 
and admits only canonical singularities.
\end{thm}
\begin{cor}\label{cor:main}
If the hypothesis in the theorem above holds and, in addition,
the greatest common divisor of
$r$, $c_1\cdot {\mathcal O}_X(1)$ and $c_1^2/2-c_2$ equals $1$, then
$M(H)$ is known to be projective by \cite[Thm. 6.11]{Ma:moduli2}, 
and its Kodaira dimension is zero.
\end{cor}
See Definition \ref{defn:terms} for definition of terms, {\it canonical singularities}
and {\it Kodaira dimension}.
The new point of this article is that we consider stable sheaves on Enriques
surfaces, and observe how good singularities of its moduli schemes are.
When $X$ is K3, Abelian, or Fano, $M(H)$ is non-singular and its Kodaira
dimension has been studied
(\cite[Prop. 6.9.]{Ma:moduli2}, \cite{Mk:symplectic}, \cite{Yos:abelian}).
However, $M(H)$ has singularities in general; how good they are?
Thus far, the works on formality are known on this problem.
Goldman-Millson \cite{Goldman-M:deform} discussed deformation of flat bundles;
remark that its first and second Chern classes are zero.
% [Fried, p.103, l.8]
The author does not know whether such a result is valid also for general stable
(or Hermite-Einstein) bundles or not.
Formality discusses the effect of degree-two part of defining ideal on
the moduli space; we would rather estimate the degree-two part itself.
%
%Zhang \cite{Zhang:formality} treated the formality of moduli space of semistable
%sheaves on a K3 surface (not on Enriques surfaces).
As to moduli of stable sheaves on an Enriques surface,
Kim dealt it in \cite{Kim:Enriques},
and obtained a finite birational map from $M(H)$ to a Lagrangian subvariety
of the moduli space of sheaves on K3 surfaces.
The singularities of $M(H)$ were not considered.\par
%
% terminal sing.ならcrepant resolutionはないので、この段落ははずしておく。
% この段落が問題として成立するのは、canonical and not terminalな時だけ。
%Here is a subject remaining to be seen.
%The Kodaira dimension of $M(H)$ is zero and some positive multiple of its canonical 
%class is trivial when assumptions in Corollary \ref{cor:main} holds.
%By \cite{Bogomolov:decomp}, any nonsingular scheme with the trivial canonical class
%has the Bogomolov decomposition.
%Are there a crepant resolution $\tilde{M}$ of $M(H)$?
%If $\tilde{M}$ exists, then there is its \'{e}tale cover $Y$ with trivial
%canonical class.
%% [BPV, p.54, Cyclic covering]
%Which appear, Abelian varieties, Calabi-Yau manifolds, 
%irreducible symplectic varieties or their products in the Bogomolov decomposition
%of $Y$?
%Here we remark that, when $r$ is coprime to $d$, $M(H)$ is non-singular,
%and so we do not need to consider its crepant resolution, and
%we only have to consider the Bogomolov decomposition of $Y$.
%Oguiso and Schroeer \cite{Oguiso:Enriques} showed that,
%if $X$ is an Enriques surface and $n\geq 2$, then $\Hilb^n(X)$ 
%has an \'{e}tale covering which is a Calabi-Yau manifold.
%Recall that any Enriques surface $X$ admits an elliptic fibration over $\PP^1$
%by \cite[Sect. 17]{BPV:text}. Let $f$ be the fiber class.
%When $r$ and $c_1\cdot f$ are coprime and $H$ is sufficiently close to $f$, 
%$M(H)$ is birational to $\Hilb^n(X)$ from \cite[p. 204]{Frd:holvb}.
%What does it happen when $r$ and $c_1\cdot f$ are not coprime?

%
\begin{acknowledgment}
The author would like to express sincere gratitude to 
Prof. M. Inaba, Prof. D. Matsushita, Prof. S. Mukai, and
Prof. Y. Namikawa for their invaluable suggestions and comments.
Deep appreciation also goes to Prof. O. Fujino and Prof. S. Mori
for explaining the importance of singularities to the author.
\end{acknowledgment}
\begin{notation}
All schemes are of finite type over $\CC$, but it seems that the greater
part of this article holds also for algebraically closed field $k$
with $(r, \operatorname{char}(k))=1$.
For sheaves $E$ and $F$ on a projective scheme over $\CC$,
$\hom(E,F)$ and $\ext^i(E,F)$ mean $\dim\Hom(E,F)$ and $\dim\Ext^i(E,F)$
respectively.
For homomorphisms $f:F\rightarrow G$ and $g:G'\rightarrow E$,
$f^*$ and $f_*$ mean natural pull-back and push-forward homomorphisms
of Exts, respectively.
For a line bundle $L$, $\Ext^i(E,E\otimes L)^{\circ}$ denotes 
the kernel of trace map $\Ext^i(E,E\otimes L)\rightarrow H^i(L)$.
\end{notation}

\section{Proof of Theorem}

Let $X$ be a non-singular complex projective surface with arbitrary Kodaira dimension.
Let us begin with recalling the definition of some terms.
\begin{defn}
For a nonzero torsion-free sheaf $F$, we denote $\chi(F(nH))/r(F)$ by $P(F(n))$.
A coherent sheaf $E$ on $X$ is {\it stable} if $E$ is torsion free
and for every proper coherent subsheaf $F$ of $E$ we have that
\[  P(F(n)) < P(E(n)) \qquad (n\gg 0).\]
\end{defn}

\begin{defn}\label{defn:terms}
(1) Given any variety $V_0$, define its {\it Kodaira dimension} $\kappa(V_0)$ to
be $\max \{ \dim \Phi_{mK_{\tilde{V}}}\bigm| m\in \NN \}$, where $\tilde{V}$ is 
a desingularization of a completion of $V_0$. Kodaira dimension is 
birational invariant. \\
(2) A normal variety $V$ is said to admit only {\it canonical singularities} when
(i) $K_V$ is $\QQ$-Cartier, and (ii) if $\phi: \tilde{V} \rightarrow V$ is a
desingularization with except divisors $E_i$, then
\[ K_{\tilde{V}} = \phi^* K_V + \textstyle \sum_i a_i E_i \qquad (a_i \geq 0). \]
When $V$ does so and $V$ is complete, $\kappa(V)$ equals 
$ \max \{ \dim \Phi_{mK_{\tilde{V}}}\bigm| m\in \NN \}$, so
we need not consider its desingularization $\tilde{V}$ in calculating $\kappa(V)$.
\end{defn}
%
%\begin{defn}
%A normal variety $V$ has only {\it canonical singularities} 
%if it satisfies the following two conditions:\\
%(a) the Weil divisor $rK_V$ is Cartier for some integer $r\geq 1$; \\
%(b) if $f: W\rightarrow V$ is a resolution of $V$ and $\{ E_i\}$ the family of
%all exceptional prime divisors of $f$, then 
%\[   K_W= f^*(K_V) +\sum a_i E_i \qquad\text{ with } a_i\geq 0.\]
%\end{defn}
%
%For hyperplane singularities, the following is known about canonical singularities:
\begin{fact}[\cite{Ishii:weighted}, Cor. 1.7, \cite{WatanabeH:ellipt-sing},\cite{Reid:young}]\label{fact:canonical}
Let $(X,x)\subset(\CC^{n+1},0)$ be a hypersurface singularity defined by 
$t_0^{a_0}+t_1^{a_1}+\dots +t_n^{a_n}$.
This singularity is canonical if and only if $\sum^n_{i=0}\frac{1}{a_i}>1$.
\end{fact}
As to singularities of $M(H)$, we shall use the following fact in Kuranishi theory:
\begin{fact}\cite{Lau:Massey}\label{fact:moduli-lci}
Let $E$ be a stable sheaf on arbitrary non-singular surface.
Denote $\dim\Ext^1(E,E)=d+b$ and $\dim\Ext^2(E,E)^{\circ}=b$, and
let $f_1, \dots, f_b$ be a basis of $\Hom(E,E(K_X))^{\circ}$.
Then the local ring $R$ of $M(H)$ at $E$ is isomorphic to
$\CC[[t_1,\dots, t_{d+b}]]/(F_1,\dots, F_b)$, where
$F_i$ is a power series starting with degree-two term, which comes from
\begin{equation}\label{eq:Ff}
 F_f: \Ext^1(E,E) \otimes \Ext^1(E,E) \longrightarrow \Ext^2(E,E) 
\longrightarrow \CC  
\end{equation}
defined by $F_f(\alpha\otimes \beta)=\operatorname{tr}(f\circ\alpha\circ\beta
+ f\circ\beta\circ\alpha)$, and its dual map
\begin{equation}\label{eq:Ffvee}
 F^{\vee}_f: \CC \longrightarrow \Ext^1(E,E)^{\vee}\otimes 
\Ext^1(E,E)^{\vee} \longrightarrow \operatorname{Sym}^2(\Ext^1(E,E)^{\vee}),
\end{equation}
where $f$ equals $f_i$.
\end{fact}
From now on, we shall suppose that $K_X$ is numerically equivalent to zero, and
let $E$ be a stable sheaf on $X$.
If $M(H)$ is singular at $E$, then $\ext^2(E,E)=\hom(E,E(K_X))=1$ since
stability of $E$ implies that $\hom(E,E(K_X))\leq 1$ in this case,
and so $\Hom(E,E(K_X))$ is spanned by one element, say $f$.
\begin{lem}\label{lem:enough-to-be}
Let us choose coordinates $t_i$ at Fact \ref{fact:moduli-lci} so that
$F_f^{\vee}=t_1^2+\dots +t_N^2$ with some $0\leq N\leq d+b$.
If the codimension of
\[ \operatorname{Sing}(M(H)):=\{ [E] \bigm| \ext^2(E,E)^{\circ}\neq 0 \} \]
in $M(H)$ is more than one and $N\geq 3$, then the singularity of $M(H)$ at $E$
is canonical.
\end{lem}
\begin{proof}
The assumption implies that $F_f^{\vee}\neq 0$, and thus $M(H)$ is l.c.i.
by Fact \ref{fact:moduli-lci}.
Since $M(H)$ is smooth in codimension one, $M(H)$ is normal and 1-Gorenstein.
In order to determine whether the singularity of l.c.i. ring $R$ at $x\in \Spec(R)$ 
is canonical or not, it is sufficient to look at its completion at $x$.
This is because the singularities of $R$ are canonical if and only if
they are rational singularities(\cite[Cor. 5.2.15]{Ishii:text}),
and a variety $V$ has only rational singularities if and only if
$V^{\operatorname{an}}$ does so (\cite[Cor. 5.11]{KM:birat}).
Since rationality of singularities are stable under small deformation
(\cite[Sect. 8]{Ishii:text}), this lemma follows Fact \ref{fact:canonical}.
\end{proof}
As to the covering $\pi: \tilde{X}\rightarrow X$ of $X$,
we have a free-action of a finite group $G$ with $|G|=d$ on $\tilde{X}$ such that
$\pi$ equals to the quotient map by $G$.
It holds that $\pi_*({\mathcal O}_{\tilde{X}})=\oplus^{d-1}_{i=0}K_X^{-i}$.
%We shall begin with two lemmas.
%
\begin{lem}\label{lem:plbk-stable}
If $E$ is a $H$-stable sheaf on $X$, 
then $\pi^*(E)$ is $\pi^* H$-semistable.
\end{lem}
\begin{proof}
Otherwise, $\pi^* E$ has a non-trivial Harder-Narashimhan filtration (HNF)
with respect to $\pi^* H$.
Let $F$ be the first part of HNF.
By the uniqueness of HNF, $F$ has a natural structure of $G$-subsheaf of $\pi^*(E)$, 
and so $F$ descends to a subsheaf $F_0$ of $E$. This $F_0$ breaks the stability
of $E$, since $K_X$ is numerically equivalent to $0$.
\end{proof}
\begin{lem}\label{lem:ExtG}
The natural map $\Ext^1_X(E,E)\rightarrow \Ext^1_{\tilde{X}}(\pi^*(E),\pi^*(E))$
induces the isomorphism
$\Ext^1(E,E)\simeq \Ext^1(\pi^*(E),\pi^*(E))^G$.
\end{lem}
\begin{proof}
Let $E_{\bullet}\rightarrow E$ be a locally-free resolution of $E$.
 \begin{multline*}
\Ext^1(\pi^*(E), \pi^*(E))= H^1(\bR\Gamma_{\tilde{X}}(\pi^*(E_{\bullet}^{\vee}\otimes
E_{\bullet}))) \\
= H^1(\bR\Gamma_X \bR\pi_*(\pi^*(E^{\vee}_{\bullet}\otimes E_{\bullet}))) 
= H^1(\bR\Gamma_X(E_{\bullet}^{\vee}\otimes E_{\bullet}\otimes \pi_*
{\mathcal O}_{\tilde{X}})) \\
= \oplus_{i=0}^{d-1} H^1(\bR\Gamma_X(E_{\bullet}^{\vee}\otimes E_{\bullet} 
\otimes K_X^{-i}))= \oplus_{i=0}^{d-1} \Ext^1(E, E\otimes K_X^{-i}).
 \end{multline*}
Since $\pi_*({\mathcal O}_{\tilde{X}})^G= \Ox$, this lemma holds.
\end{proof}
Now, let $f:E\rightarrow E(K_X)$ be a nonzero traceless homomorphism. 
Since $K_X$ is numerically equivalent to zero and $E$ is stable, 
this $f$ is isomorphism and so $\det(f)\neq 0$.
Fix an isomorphism ${\mathcal O_{\tilde{X}}}\simeq \pi^*(K_X)$,
and let $t\in \Gamma(\pi^*(K_X))$ be the image of $1$ by this.
When we denote the eigenpolynomial of $\pi^*f$ by $P_{\pi^*f}(t)$,
we can decompose it into eigenvalues 
\[ P_{\pi^*f}(t)= \prod_i (t-a_i)^{n_i}, \]
where $a_i$ are elements in $H^0(\Ktx)$ which differ from each other, 
from the fact that $\Ktx=\Otx$.
Let us fix $a_1$, and pick any $g\in G$.
Then $\det(f)\neq 0$ implies $a_1\neq 0$ , and
$g(\pi^*(f))=\pi^*(f)$ implies $g(a_1)=a_i$ with some $i$.
\begin{lem}\label{lem:FreeAct-a1}
 If $g\neq e$, then $g(a_1)\neq a_1$.
\end{lem}
\begin{proof}
Otherwise, one can indicate the orbit $O_{a_1}$ as $\{a_1,\dots, a_m \}$
with some $m<|G|=d$. Then $a_1\cdot \dots \cdot a_m$ lies in
$\Gamma(m\Ktx)^G=\Gamma(mK_X)$, but we know that $\Gamma(mK_X)=0$ for
$m<d$. This contradicts to the fact that $\det(f)\neq 0$.
\end{proof}
\begin{lem}\label{lem:fG=0}
$ \prod_{g\in G}(\pi^*f-g(a_1))=0$ in $\Hom(\pi^* E, \pi^* E(d K_X)))$.
\end{lem}
\begin{proof}
This homomorphism $f^G$ lies in
$\Hom(\pi^* E,\pi^* E(d K_X))^G=\Hom(E,E(d K_X))$.
If $f^G\neq 0$, then $f^G$ should be injective since $E$ is stable.
However, $f^G$ is not injective obviously.
\end{proof}
For $g\in G$, denote $\Ker(\pi^*f-g(a_1))\subset \pi^* E$ by $F_g$.
%Remark that $F_g\neq 0$ and that $P(F_g(n))=P(E(n))$ by 
%Lemma \ref{lem:plbk-stable}. 
\begin{lem}\label{lem:eigen-decomp}
The natural map $\oplus_{g\in G} F_g \rightarrow \pi^* E$ is isomorphic.
\end{lem}
\begin{proof}
Let us consider the map
\begin{equation}\label{eq:Fe-Qe}
 \prod_{g\neq e} (\pi^*f-g(a_1)) : F_e \hookrightarrow \pi^*E \rightarrow
 F'_e:=\IIm\prod_{g\neq e} (\pi^*f-g(a_1)). 
\end{equation}
If $\alpha \in F_e$ belongs to the kernel of this map, then
$0=\prod_{g\neq e} (\pi^*f-g(a_1))(\alpha)= \prod_{g\neq e}(a_1-g(a_1))(\alpha)$,
and then Lemma \ref{lem:FreeAct-a1} and $\Ktx=\Otx$ deduces that $\alpha=0$.
Thus the map \eqref{eq:Fe-Qe} is injective.
Next, 
the map $f^{-(d-1)}: \pi^* E((d-1)K_X) \rightarrow \pi^* E$ induces the map
$f^{-(d-1)}: F'_e \rightarrow F_e$ by Lemma \ref{lem:fG=0}.
One can check that 
$f^{-(d-1)}\circ \prod_{g\neq e} (\pi^*f-g(a_1)) :F_e \rightarrow F_e$
is the multiple map by a non-zero constant, and so is isomorphic
by Lemma \ref{lem:FreeAct-a1}.
Thereby, the map \eqref{eq:Fe-Qe} is bijective, and hence 
we have that $\pi^* E= F_e \oplus (\pi^* E/F_e)$.
Repeating this, we can show this lemma.
\end{proof}
Since $F_g\neq 0$, $P(F_g(n))=P(\pi^*E(n))$ by Lemma \ref{lem:plbk-stable}
and Lemma \ref{lem:eigen-decomp}.
The homomorphism $\pi^*f-g(a_1)$ induces exact sequences
\begin{align}\label{eq:1st-exseq1}
 0 \longrightarrow F_g \overset{i_g}{\longrightarrow} & \pi^*E 
 \overset{p_g}{\longrightarrow} G_g
 \longrightarrow 0, \qquad \text{and} \\ \label{eq:1st-exseq2}
 0 \longrightarrow G_g \overset{j_g}{\longrightarrow} & \pi^*E(K_X) 
 \overset{q_g}{\longrightarrow} Q_g \longrightarrow 0. 
\end{align}
\begin{prop}\label{prop:dimIm}
One can find $h\in G$ with $h\neq e$ as follows:
the dimension of the image of the natural map, 
\begin{equation}\label{eq:Def-Im}
  \Ext^1(Q_e(-K_X),F_h) \overset{i_{h*}}{\longrightarrow} \Ext^1(Q_e(-K_X), \pi^*E)
   \overset{q_e^*}{\longrightarrow} \Ext^1(\pi^*E, \pi^* E),
\end{equation}
which comes from the exact sequences \eqref{eq:1st-exseq1} and \eqref{eq:1st-exseq2}, 
is $-\frac{1}{d}\chi(E,E)$ or more.
%is $-\frac{1}{d}\chi(E,E)-d^2+1$ or more.
\end{prop}
We shall prove this proposition later; we presume it proved for now.
From now on, suppose that the hypothesis in Theorem \ref{thm:main} holds,
which implies that $-\chi(E,E)/d -d^2+1 \geq 3$. Hence
there is a nonzero element $\alpha\in\Ext^1(\pi^*E, \pi^* E)$ 
lying in the image of the map \eqref{eq:Def-Im}.
For any $h\in G$, we have a following commutative diagram
\begin{equation}\label{eq:comdiag}
\xymatrix{
\Ext^1(Q_e(-K_X),F_h) \ar[r]_{i_{h*}} \ar[d]^{q_e^*} &
\Ext^1(Q_e(-K_X),\pi^*E)  \ar[d]_{q_e^*} & \\ 
\Ext^1(\pi^*E,F_h) \ar[r]_{i_{h*}} & \Ext^1(\pi^*E, \pi^*E) \ar[r]_{p_{h*}} 
\ar[d]^{j_e^*} &
\Ext^1(\pi^*E, G_h) \ar[d]^{j_{h*}} \\
  & \Ext^1(G_e(-K_X),\pi^*E) \ar[d]^{p^*_e} & \Ext^1(\pi^* E,\pi^*E(K_X)) \\
   & \Ext^1(\pi^*E(-K_X),\pi^* E), & }
\end{equation}
and it holds that $p_e^*\circ j_e^*=(\pi^* f-a_1)^*$ and
$j_{h*}\circ p_{h*}=(\pi^*f-h(a_1))_*$.
Thus $(\pi^*f -a_1)^*(\alpha)=0$ and $(\pi^*f-h(a_1))_*(\alpha)=0$.
From the definition of pull-back and push-forward, it implies that
\begin{equation}\label{eq:alpha-seisitu}
\alpha\circ\pi^*f= (\pi^*f)^*(\alpha)= a_1\alpha \qquad\text{and}\quad
\pi^* f\circ\alpha=(\pi^* f)_*(\alpha)=h(a_1)\alpha
\end{equation}
in $\Ext^1(\pi^* E, \pi^* E(K_X))$.
Here we denote $\sum_{g\in G} g(\alpha) \in \Ext^1(\pi^*(E), \pi^*(E))$
by $\alpha^G$.
If $\alpha^G =0$, then $-\alpha\in \Ker(\pi^* f-a_1)$ and 
$\sum_{g\neq e}g(\alpha)$ should belong
to $\Ker(\prod_{g\neq e}(\pi^*f-g(a_1)))$, but
$\Ker(\pi^* f-a_1)\cap\Ker(\prod_{g\neq e}(\pi^*f-g(a_1)))=0$ 
by Lemma \ref{lem:FreeAct-a1}.
Thus $\alpha^G\neq 0$.
From Lemma \ref{lem:ExtG}, $\alpha^G$ descends to a nonzero element
$\bar{\alpha}\in\Ext^1(E,E)$.
Since $f$ is isomorphic, by the Serre duality
we have some $\beta\in\Ext^1(E,E)$ such that
\begin{equation}\label{eq:beta}
\tr(f\circ \bar{\alpha}\circ \beta)\neq 0.
\end{equation}
\begin{lem}
If $\beta$ satisfies \eqref{eq:beta}, then 
$\tr(f\circ \bar{\alpha}\circ \beta+ f\circ\beta\circ\bar{\alpha})\neq 0$.
\end{lem}
\begin{proof}
From the definition of $\bar{\alpha}$, we have
\begin{multline*}
\tr(\pi^*(f\circ\bar{\alpha}\circ \beta))=
\tr(\pi^* f \circ \pi^*\bar{\alpha} \circ \pi^*\beta)\\
= \sum_{g\in G} \tr(\pi^*f \circ g(\alpha)\circ \pi^*\beta )
= \sum_g \tr (gh(a_1)\cdot g(\alpha)\circ\pi^*(\beta)),
\end{multline*}
where one get the last equation by the action of $g$
on the second at \eqref{eq:alpha-seisitu}.
From a basic property of trace \cite[p. 257]{HL:text}, it holds that
\begin{multline*}
-\tr(\pi^*(f\circ\beta\circ\bar{\alpha}))
= \tr (\pi^*(\bar{\alpha}\circ f\circ \beta)) \\
= \sum_g \tr(g(\alpha)\circ \pi^* f \circ \pi^*\beta)
= \sum_g \tr(g(a_1)\cdot g(\alpha)\circ\pi^*(\beta)),
\end{multline*}
where one get the last equation by the action of $g$
on the first at \eqref{eq:alpha-seisitu}.
Since $a_1\in \Gamma(\pi^* K_X)=\Gamma(\Otx)$ is nowhere vanishing,
we can set $\lambda=h(a_1)/a_1 \in \Gamma(\Otx)\simeq \CC \simeq \Gamma(\Otx)^G$,
and thus
$gh(a_1)/g(a_1)= g(h(a_1)/a_1)= h(a_1)/a_1= \lambda$ for any $g\in G$.
Thereby
$\tr(\pi^*(f\circ\bar{\alpha}\circ\beta))= 
-\lambda\tr(\pi^*(f\circ\beta\circ\bar{\alpha}))$,
and then
\[\pi^*\tr(f\circ\bar{\alpha}\circ\beta+ f\circ\beta\circ\bar{\alpha})
= (1-\frac{1}{\lambda})\pi^* \tr(f\circ \bar{\alpha}\circ\beta)\neq 0,\]
since $\lambda\neq 0,1$ by Lemma \ref{lem:FreeAct-a1}.
\end{proof}
Summing up, we have found $\bar{\alpha}$ and $\beta$ in $\Ext^1(E,E)$ such
that $F_f(\bar{\alpha},\beta)\neq 0$ for $F_f$ at \eqref{eq:Ff}
if $-\frac{1}{d}\chi(E,E)>0$  by Proposition \ref{prop:dimIm}, and
%if $-\frac{1}{d}\chi(E,E)-d^2+1>0$  by Proposition \ref{prop:dimIm}, and
consequently we can find $t_1\in \Ext^1(E,E)$ such that
$F_f(t_1,t_1)=1$.
Then, replacing the argument above from $\Ext^1(E,E)$ to
$\Ker F_f(t_1, \cdot )$, we can find $t_2$ such that
$F_f(t_1, t_2)=0$ and $F_f(t_2, t_2)=1$, when
$-\frac{1}{d}\chi(E,E)\geq 2$.
%$-\frac{1}{d}\chi(E,E)-d^2+1\geq 2$.
%
\begin{lem}\label{lem:dimSing}
(a) If $E$ belongs to $\operatorname{Sing} (M(H))$ at Lemma \ref{lem:enough-to-be},
then $E\simeq \pi_*(F_g)$ for all $g\in G$. \\
(b) For any $g\in G$, $F_g$ is $\pi^* H$-stable. \\
(c) When $-\frac{1}{d}\chi(E,E)\geq 3$, the codimension of 
%When $-\frac{1}{d}\chi(E,E)-d^2+1\geq 3$, the codimension of 
$\operatorname{Sing}(M(H))$ in $M(H)$ is greater than $1$.
\end{lem}
\begin{proof}
Since $\oplus_{g\in G} F_g \simeq \pi^* E$, we have
an injective map $E\rightarrow \pi_*(F_e) $. 
Applying the Grothendieck-Riemann-Roch theorem and projection formula to $\pi_*(F_e)$,
we have $d\ch(E)\equiv \sum_g\ch(\pi_*(F_g)))= d\ch(\pi_*(F_e))$ 
in $\operatorname{CH}(X)_{\QQ}$,
since $F_g=g^*(F_e)$ for all $g$.
Thus the injective map $E\rightarrow \pi_*(F_e) $ is isomorphic.
Clearly $\pi_*(F_g)=\pi_*(F_e)$, so we get (a).
Next, $F_e$ satisfies $P(F_e(n\pi^*H))=P(\pi^*E(nH))$ by Lemma \ref{lem:plbk-stable}.
If $F_e$ is not $\pi^* H$-stable, then there is a proper subsheaf $F'\subset F_e$
such that $P(F'(n\pi^*H))=P(\pi^* E(nH))$. Then $\oplus_{g\in G} g^*(F') \subset \pi^*(E)$
descends to a proper subsheaf of $E$, which becomes a $H$-destabilizer of $E$.
This is contradiction and deduces (b).
Next, as to (c), we can assume $d\geq 2$, for if $d=1$ then $K_X$ is trivial
and then $M(H)$ is nonsingular by \cite{Mk:symplectic}.
From (a), the moduli number of
$\operatorname{Sing}(M(H))$ is not greater than
$\ext^1_{\tilde{X}}(F_e,F_e)= -\chi(F_e,F_e)+2$, which is
not greater than $-\chi(E,E)/d +2$ from equations \eqref{eq:ext-E-FeFg} and
\eqref{eq:ext-Fe-FeFg} below.
Thereby when $-\frac{1}{d}\chi(E,E)\geq 3$ we can check that
\begin{multline*}
 \dim M(H)-\dim\operatorname{Sing}(M(H))\geq \ext^1(E,E)-\ext^2(E,E)-\ext^1(F_e,F_e)\\
\geq -\chi(E,E)+1-2+ \frac{1}{d}\chi(E,E) = \frac{1-d}{d}\chi(E,E)-1 \geq 3(d-1)-1\geq 2
\end{multline*}
and this leads to (c).
\end{proof}
In such a way, we can describe $F^{\vee}_f$ at \eqref{eq:Ffvee} as
$F^{\vee}=t_1^2+\dots + t_N^2$ with $N\geq 3$ when
$-\frac{1}{d}\chi(E,E)\geq 3$.
%$\ext^1(E,E)^{\circ}-\ext^2(E,E)^{\circ}=-\chi(E,E)+\chi(\Ox) \geq 3d+\chi(\Ox)$.
%$-\frac{1}{d}\chi(E,E)-d^2+1\geq 3$,
Therefore Theorem \ref{thm:main} follows
from Lemma \ref{lem:enough-to-be}, Lemma \ref{lem:dimSing} 
and Proposition \ref{prop:dimIm}.
We know that
some positive multiple of the canonical class of $M(H)$ 
equals ${\mathcal O}_{M(H)}$ by Grothendieck-Riemann-Roch's theorem, 
so we also have Corollary \ref{cor:main}.\\
  \par
{\it Proof of Proposition \ref{prop:dimIm}:}
From Lemma \ref{lem:eigen-decomp},
\begin{equation}\label{eq:ext-E-FeFg}
 |G| \chi(E,E) = \chi(\pi^*E, \pi^*E)= \sum_{g\in G}\sum_{g'\in G}
\chi(F_{g'}, F_g)= |G|\sum_{g\in G} \chi(F_e, F_g),
\end{equation}
and so $\chi(E,E) =\sum_g \chi(F_e, F_g)$.
From the Riemann-Roch theorem, we have
\begin{multline}\label{eq:FeFg}
\chi(F_e, F_g) =
(r(F_g)c_1(F_e)^2+r(F_e)c_1(F_g)^2)/2 \\
-r(F_e)c_2(F_g)-r(F_g)c_2(F_e) -c_1(F_e)c_1(F_g)  +r(F_e)r(F_g)\chi(\Otx).
%\leq -2r(F_e)c_2(F_e)+r(F_e)^2\chi(\Otx)=\chi(F_e,F_e).
\end{multline}
Since $F_g=g^*(F_e)$, it holds that $r(F_e)=r(F_g)$, $c_1(F_e)^2=c_1(F_g)^2$,
and $c_2(F_e)=c_2(F_g)$.
By the Hodge index theorem
$(c_1(F_e)-c_1(F_g))^2=c_1(F_e)^2+c_1(F_g)^2-2c_1(F_e)c_1(F_g)=
2(c_1(F_e)^2-c_1(F_e)c_1(F_g))\leq 0$, and hence \eqref{eq:FeFg} says that
\begin{equation}\label{eq:ext-Fe-FeFg}
 \chi(F_e,F_e)-\chi(F_e,F_g)= -c_1(F_e)^2+c_1(F_e)c_1(F_g) \geq 0.
\end{equation}
Then some $h\neq e$ satisfies that $\chi(F_e,F_h)\leq \chi(E,E)/d$;
otherwise, all $g$ satisfy that $\chi(F_e, F_g)> \chi(E,E)/d$, 
which contradicts to the fact that
$\chi(E,E) =\sum_g \chi(F_e, F_g)$.
For such $h\neq e$, 
$\ext^1(F_e,F_h)\geq -\frac{1}{d}\chi(E,E).$
%\[ \ext^1(F_e,F_h)\geq \hom(F_e,F_h)+\hom(F_h,F_e(K_X))-\frac{1}{d}\chi(E,E).\]
%
%The sheaf $F_e$ is $\pi^*H$-stable by Lemma \ref{lem:dimSing}. 
%Hence
%\begin{equation}\label{eq:hom-eg}
%\hom(F_e(-K_X),F_h)\leq 1. 
%\end{equation}
As to the homomorphism \eqref{eq:Def-Im}, remark that $Q_e=F_e(K_X)$ and
that exact sequences
\eqref{eq:1st-exseq1} and \eqref{eq:1st-exseq2} split since $\pi^*(K_X)$
is trivial. Therefore the following holds:
%one can deduce from the diagram
%\eqref{eq:comdiag} that
% 9/4 Note p.7
\begin{multline*}
 \dim\IIm( \Ext^1(Q_e(-K_X), F_h) \rightarrow \Ext^1(\pi^*E, \pi^* E)) \\
 \geq \ext^1(Q_e(-K_X), F_h)= \ext^1(F_e, F_h)\geq  -\frac{1}{d}\chi(E,E).
%
% \ext^1(Q_e(-K_X), F_h)- \hom(G_e(-K_X),\pi^* E) -\hom(Q_e(-K_X), G_h) \\
% = \ext^1(Q_e(-K_X), F_h)-\sum_{g\neq e} \sum_{g'} \hom(F_g,F_{g'})
%  -\sum_{g'\neq h} \hom(F_e,F_{g'})) \\
% \geq \ext^1(F_e,F_h)-(d-1)d  -(d-1) \\
%\geq \ext^1(F_e,F_h) -d^2+1 \geq -\frac{1}{d}\chi(E,E) -d^2+1
\end{multline*} 
%in the same way to \eqref{eq:hom-eg}, since $G_e=\oplus_{g\neq G} F_g(K_X)$ and
%$Q_e=F_e(K_X)$. Therefore we obtain Proposition \ref{prop:dimIm}.

%%%%%

%\nocite{Ishii:text} 
%\bibliography{mybib.bib}

\providecommand{\bysame}{\leavevmode\hbox to3em{\hrulefill}\thinspace}
\providecommand{\MR}{\relax\ifhmode\unskip\space\fi MR }
\providecommand{\MRhref}[2]{%
  \href{http://www.ams.org/mathscinet-getitem?mr=#1}{#2}
}
\providecommand{\href}[2]{#2}

\end{document}